\documentclass[12pt]{amsart}
\usepackage{amssymb}
\usepackage{amsthm,amsfonts}

\usepackage{amsfonts}
\usepackage{amsmath}
\usepackage[abbrev]{amsrefs}
\usepackage{enumerate}
\usepackage{color}
\usepackage[colorlinks,linkcolor=blue]{hyperref}
\usepackage[nameinlink]{cleveref}

\newtheorem{Theorem}{Theorem}[section]

\newtheorem{Lemma}[Theorem]{Lemma}
\newtheorem{Corollary}[Theorem]{Corollary}
\newtheorem{Proposition}[Theorem]{Proposition}
\newtheorem{Definition}[Theorem]{Definition}

\newtheorem{corollary}[Theorem]{Corollary}
\newtheorem{proposition}[Theorem]{Proposition}
\newtheorem{definition}[Theorem]{Definition}

\theoremstyle{definition}

\begin{document}
	
	\title[Remark on the Daugavet property for complex Banach spaces]{Remark on the Daugavet property for complex Banach spaces}
	
	
	\author{Han Ju Lee}
	\address{Department of Mathematics Education, Dongguk University - Seoul, 04620 (Seoul), Republic of Korea}
	\email{hanjulee@dgu.ac.kr}
	
	\author{Hyung-Joon Tag*}
	\address{Department of Mathematics and Statistics, Sejong University, 05006 (Seoul), Republic of Korea}
	\email{hjtag4@gmail.com}
	\date{\today}
	
\maketitle

\begin{abstract}
In this article, we study the Daugavet property and the diametral diameter two properties in complex Banach spaces. The characterizations for both Daugavet and $\Delta$-points are revisited in the context of complex Banach spaces. We also provide relationships between some variants of alternative convexity and smoothness, nonsquareness, and the Daugavet property. As a consequence, every strongly locally uniformly alternatively convex or smooth (sluacs) Banach space does not contain $\Delta$-points from the fact that such spaces are locally uniformly nonsquare. We also study the convex diametral local diameter two property (convex-DLD2P) and the polynomial Daugavet property in the vector-valued function space $A(K, X)$. From an explicit computation of the polynomial Daugavetian index of $A(K, X)$, we show that the space $A(K, X)$ has the polynomial Daugavet property if and only if either the base algebra $A$ or the range space $X$ has the polynomial Daugavet property. Consequently, we obtain that the polynomial Daugavet property, Daugavet property, diameteral diameter two properties, and property ($\mathcal{D}$) are equivalent for infinite-dimensional uniform algebras. 
\end{abstract}

\section{Introduction}
In the theory of Banach spaces, various properties that are related to certain behaviors of vector measures and bounded linear operators have been studied from a geometrical point of view. We focus on Daugavet property and diametral diameter two properties in this article. Let $X$ be a Banach space on $\mathbb{F} = \mathbb{R}$ or $\mathbb{C}$. For a Banach space $X$, we denote the unit ball and the unit sphere of $X$ by $B_X$ and $S_X$, respectively. The dual space of $X$ is denoted by $X^*$. For $\epsilon > 0$, $x^* \in S_{X^*}$, and $x \in S_X$, a slice $S(x^*, \epsilon)$ of the unit ball $B_X$ is defined by $S(x^*, \epsilon)=\{x \in B_X: \text{Re}\,x^*x > 1 - \epsilon\}$ and a weak$^*$-slice $S(x, \epsilon)$ of $B_{X^*}$ by $S(x, \epsilon) = \{x^* \in B_{X^*}: \text{Re}\,x^*x > 1 - \epsilon\}$.

 Many results on Banach spaces with the Radon-Nikod\'ym property have been obtained from this perspective. Even though its definition was originally given in terms of vector measures, it is now well-known that the Banach spaces with the Radon-Nikod\'ym property can be characterized by the existence of slices with arbitrarily small diameter as well as the existence of denting points. For more details on the geometrical aspect of the Radon-Nikod\'ym property and its application to other research topics in Banach spaces, we refer to \cite{DU}.  

A Banach space $X$ is said to have the {\it Daugavet property} if every rank-one operator $T: X \rightarrow X$ satisfies the following equation:
\[
\|I + T\| = 1 + \|T\|.
\] 
We call this equation the {\it Daugavet equation}. The spaces $C(K)$, where $K$ is a compact Hausdorff space without isolated points, $L_1(\mu)$, and $L_{\infty}(\mu)$ with a nonatomic measure $\mu$, are classical examples with the Daugavet property. The infinite-dimensional uniform algebras also have the Daugavet property if and only if their Shilov boundaries do not have isolated points \cite{Wo, LT}.  Moreover, the Daugavet property in Musielak-Orlicz spaces \cite{KK}, in Lipschitz-free spaces \cite{JR}, and in rearrangement-invariant Banach function lattices \cite{AKM, KMMW} have been examined. It is well-known that every slice of $B_X$ has diameter two if $X$ has the Daugavet property, which tells us that Banach spaces with this property are on the opposite spectrum to the Radon-Nikod\'ym property.

The following characterization allows us to study the Daugavet property with slices.
  
\begin{Lemma}\cite[Lemma 2.2]{KSSW}
	\label{lem:Daug}
	The following are equivalent.
	\begin{enumerate}[{\rm(i)}]
		\item A Banach space $(X,\|\cdot\|)$ has the Daugavet property.
		\item\label{Daugii} For every slice $S = S(x^*,\epsilon)$ and $x \in S_X$, there exists $y\in S_{X}\cap S$ such that $\|x+y\|>2-\epsilon$.
		\item\label{Daugiii} For every weak$^{*}$-slice $S^* = S(x,\epsilon)$ and every $x^* \in S_{X^*}$, there exists $y^*\in S_{X^*}\cap S^*$ such that $\|x^*+y^*\|>2-\epsilon$.
		\end{enumerate}
	\end{Lemma}
 
Later, the diametral diameter two properties (diametral D2Ps) and the property ($\mathcal{D}$) have gained attention \cite{AHLP, W} from many researchers. They are known to be weaker than the Daugavet property.
 
\begin{Definition}
	\begin{enumerate}[\rm(i)]
		\item A Banach space $X$ has property ($\mathcal{D}$) if for every rank-one, norm-one projection $P: X \rightarrow X$ satisfies $\|I - P\| = 2$.
		\item A Banach space $X$ has the diametral local diameter two property (DLD2P) if for every slice $S$ of the unit ball, every $x \in S \cap S_X$, and every $\epsilon > 0$ there exists $y \in S$ such that $\|x - y\| \geq 2 - \epsilon$.
		\item A Banach space $X$ has the diametral diameter two property (DD2P) if for every nonempty weakly open subset $W$ of the unit ball, every $x \in W \cap S_X$, and every $\epsilon > 0$, there exists $y \in W$ such that $\|x- y\| \geq 2 - \epsilon$.
	\end{enumerate}
\end{Definition}
\noindent The first known example that possesses the property ($\mathcal{D}$) is a certain subspace of $L_1$ constructed with martingales \cite{BR}. Later on, this space was shown to have the Daugavet property \cite{KW}. In view of \cite{IK}, every rank-one projection $P$ on a Banach space $X$ with the DLD2P satisfies $\|I - P\| \geq 2$, and so the DLD2P implies the property ($\mathcal{D}$). As a matter of fact, the property ($\mathcal{D}$) was thought to be equivalent to the DLD2P. However, since a scalar multiple of a projection is not a projection, the validity of the equivalence is not clear up to now \cite{AHLP}. The implication (iii) $\implies$ (ii) holds because every slice is a weakly open subset of the unit ball.  

The DLD2P and the Daugavet property can be also considered from a local perspective by using $\Delta$-points and Daugavet points. Let  $\Delta_{\epsilon}(x) = \{y \in B_X: \|x - y\| \geq 2 - \epsilon\}$, where $x \in S_X$ and $\epsilon > 0$.
\begin{Definition}
	\begin{enumerate}[\rm(i)]
		\item A point $x \in S_X$ is a $\Delta$-point if $x \in \overline{conv}\Delta_{\epsilon}(x)$ for every $\epsilon > 0$.
		\item A point $x \in S_X$ is a Daugavet point if $B_X = \overline{conv}\Delta_{\epsilon}(x)$ for every $\epsilon > 0$.
	\end{enumerate} 
\end{Definition}
\noindent Notice that the set $\Delta_{\epsilon}(x)$ is defined independently of the scalar fields $\mathbb{F} = \mathbb{R}$ or $\mathbb{C}$ on Banach spaces. Hence we may use the same definitions for $\Delta$-points and Daugavet points for complex Banach spaces. We denote the set of all $\Delta$-points of $X$ by $\Delta_X$. There is also a weaker property than the Daugavet property described by the set $\Delta_X$. 
\begin{Definition}
	A Banach space $X$ has the convex diametral local diameter two property (convex-DLD2P) if $\overline{conv}\Delta_X = B_X$.
\end{Definition}
It is well-known that a real Banach space $X$ has the Daugavet property (resp. the DLD2P) if and only if every point on the unit sphere is a Daugavet point (resp. a $\Delta$-point).

We mention that many recent results on the Daugavet property, the diametral D2Ps, Daugavet points, and $\Delta$-points have been mostly revolved around {\it real} Banach spaces. But there are several results concerning these concepts in complex Banach spaces; see \cite{Kad, MPR}. For a real Banach space $X$, it is well-known that  $\Delta$-points are connected to certain behaviors of slices of the unit ball and of rank-one projections. The notation $x^* \otimes x$ stands for a linear operator from $X$ to itself defined by $(x^* \otimes x)(y) = x^*y \cdot x$, where $x, y \in X$ and $x^* \in X^*$. 

\begin{Theorem}\label{th:realdelta}\cite{AHLP}
	Let $X$ be a real Banach space. Then the following statements are equivalent.
	\begin{enumerate}[\rm(i)]
		\item $x \in S_X$ is a $\Delta$-point.
		\item For every slice $S$ of $B_X$ with $x \in S \cap S_X$ and $\epsilon > 0$, there exists $y \in S$ such that $\|x - y\| \geq 2 - \epsilon$.
		\item For every rank-one projection $P = x^* \otimes x$ with $x^*x = 1$, we have $\|I - P\| \geq 2$.
	\end{enumerate}
\end{Theorem}
\noindent Even though the complex analogue of this relationship may be well-known to the specialists, we will state and prove it for completeness. In addition, while the Daugavet property for complex Banach spaces can be examined through rank-one real-linear operators \cite{KMM}, it has not been known whether we can examine the DLD2P in a similar spirit with rank-one real-projections. We also study this here.

Since a denting point is always contained in a slice of arbitrarily small diameter, such a point is neither a $\Delta$-point nor a Daugavet point. This implies that any (locally) uniformly rotund real Banach spaces cannot have $\Delta$-points. Recently, identifying the Banach spaces that do not contain these points has been an active research topic. For example, it is shown in \cite{ALMP} that every uniformly nonsquare real Banach space does not have $\Delta$-points. Furthermore, a locally uniformly nonsquare real Banach space does not have $\Delta$-points \cite{KLT}. We will examine strongly locally uniformly alternatively convex or smooth (sluacs) Banach spaces in this article. We mention that alternative convexity or smoothness are related to the anti-Daugavet property and the nonsquareness property \cite{J, WSL}.

Banach spaces that satisfy the Daugavet equation for weakly compact polynomials are also studied in \cite{CGMM, CGKM, MMP}. For Banach spaces $X, Y$, let $\mathcal{L}(^kX;Y)$ be the space of bounded  $k$-linear mappings from $X$ to $Y$ and let $\Psi_k: X \rightarrow X^k$ be a diagonal mapping defined by 
\[
\Psi_k(x) = \underset{k\,\,\, \text{times}}{\underbrace{(x, x, \dots, x)}}.
\]
A mapping is called a bounded {\it $k$-homogeneous polynomial} from $X$ to $Y$ if it is the composition of $\Psi_k$ with an element in $\mathcal{L}(^kX;Y)$. We denote the set of all bounded $k$-homogenous polynomials from $X$ to $Y$ by $\mathcal{P}(^k X;Y)$. A {\it polynomial} is a finite sum of bounded homogenous polynomials from $X$ to $Y$. We also denote the set of all polynomials from $X$ to $Y$ by $\mathcal{P}(X;Y)$ and the set of all scalar-valued continuous polynomials by $\mathcal{P}(X)$. The space $\mathcal{P}(X;Y)$ is a normed space endowed with the norm $\|P\| = \sup_{x \in B_X}\|Px\|_X$ and $\mathcal{P}(^k X;Y)$ is a closed subspace of $\mathcal{P}(X;Y)$.

We say a polynomial $P \in \mathcal{P}(X;Y)$ is weakly compact if $P(B_X)$ is a relatively weakly compact subset of $Y$. A Banach space $X$ is said to have the {\it polynomial Daugavet property} if every weakly compact polynomial $P \in \mathcal{P}(X; X)$ satisfies 
\[
\|I + P\| = 1 + \|P\|.
\]  
If $X$ has the polynomial Daugavet property, then the space also has the Daugavet property. It is also well-known that the polynomial Daugavet property can be described in terms of scalar-valued polynomials. Here the notation $p \otimes x$ stands for a polynomial from $X$ to itself defined by $(p \otimes x)(y) = p(y) \cdot x$, where $x,y \in X$ and $p \in \mathcal{P}(X)$. 

\begin{Theorem}\cite[Corollary 2.2]{CGMM}\label{th:polydauggen}
	Let $X$ be a real or complex Banach space. Then the following statements are equivalent:
	\begin{enumerate}[\rm(i)]
		\item $X$ has the polynomial Daugavet property.
		\item For every $p \in \mathcal{P}(X)$ with $\|p\| = 1$, every $x_0 \in S_X$, and every $\epsilon > 0$, there exist $\omega \in S_{\mathbb{C}}$ and $y \in B_X$ such that $\text{Re}\,\omega p(y) > 1 - \epsilon$ and $\|x_0 + \omega y\| > 2 - \epsilon$.
		\item For every $p \in \mathcal{P}(X)$ and every $x_0 \in X$, the polynomial $p \otimes x_0$ satisfies the Daugavet equation.
	\end{enumerate}	
\end{Theorem} 

In this article, we will look at the polynomial Daugavet property in a function space $A(K, X)$ over the base algebra $A$, which will be defined later. This class of function spaces includes uniform algebras and the space of Banach space-valued continuous functions on a compact Hausdorff space. The Daugavet property and the diametral D2Ps of the vector-valued function spaces $A(K, X)$ are studied in \cite{LT}. From the same article, assuming the uniform convexity of the range space $X$ and $A \otimes X \subset A(K, X)$, it is shown that the space $A(K, X)$ has the Daugavet property if and only if its base algebra also has the Daugavet property. It is also shown in \cite{CJT} that if $X$ has the Daugavet property then $A(K, X)$ has the Daugavet property. Here we attempt to find a necessary and sufficient condition for $A(K, X)$ to have the polynomial Daugavet property.

The article consists of three parts. In section 2, we revisit well-known facts about $\Delta$-points and Daugavet points in the context of complex Banach spaces. Like the Daugavet property, the DLD2P can be also analyzed by using rank-one real-projections (Theorem \ref{prop:deltaequiv}).  In Section 3, we examine the relationship between alternative convexity or smoothness and nonsquareness. From the fact that strongly locally uniformly alternatively convex or smooth (sluacs) Banach spaces are locally uniformly nonsquare (Proposition 3.8.(i)), no sluacs Banach space contains $\Delta$-points (Theorem \ref{prop:nod}). In section 4, we study the polynomial Daugavet property of the space $A(K, X)$. Here we explicitly compute the polynomial Daugavetian index of the space $A(K, X)$ (Theorem \ref{th:polydauind}). The space $A(K, X)$ has a bicontractive projection if the Shilov boundary of the base algebra $A$ has isolated points (Proposition \ref{prop:bicontractive}). As a consequence, we will show that $A(K, X)$ has the polynomial Daugavet property if and only if the base algebra $A$ or the range space $X$ has the polynomial Daugavet property (Corollary \ref{cor:polydaugAKX}).

\section{Delta-points and Daugavet points in complex Banach spaces}

In this section, we study $\Delta$-points and Daugavet points for complex Banach spaces. Although one may find that a certain portion of the proofs are similar to the real case, we include them in this article for completeness. However, we mention that the complex scalar field $\mathbb{C}$ provides something more, namely, a tool to analyze the Daugavet property and the DLD2P for complex Banach spaces through rank-one real-linear operators and rank-one real-projections, respectively. We recall the following useful lemma that holds for both real and complex Banach spaces:
 
\begin{Lemma}\cite[Lemma 2.1]{IK}\label{lem:subslice}
	Let $x^* \in S_{X^*}$, $\epsilon > 0$. Then for every $x \in S(x^*, \epsilon)$ and every $\delta \in(0, \epsilon)$ there exists $y^* \in S_{X^*}$ such that $x \in S(y^*, \delta)$ and $S(y^*, \delta) \subset S(x^*, \epsilon)$.
\end{Lemma}

\begin{Theorem}\label{prop:deltaequiv}
	Let $X$ be a complex Banach space and $x \in S_X$. Then the following statements are equivalent.
	\begin{enumerate}[\rm(i)]
		\item $x$ is a $\Delta$-point.
		\item For every slice $S$ of $B_X$ with $x \in S$ and $\epsilon > 0$, there exists $y \in S$ such that $\|x - y\| \geq 2 - \epsilon$.
		\item For every rank-one projection $P = x^* \otimes x$ with $x^*x = 1$, we have $\|I - P\| \geq 2$.
		\item For every rank-one real-projection $P = \text{Re}\,x^* \otimes x$ with $x^*x = 1$, we have $\|I - P\| \geq 2$.
		\item For every rank-one linear operator $T = \frac{x^*}{\|x^*\|} \otimes x$ with $x^*x = 1$, we have $\|I - T\| = 2$.
		\item For every rank-one real-linear operator $T = \frac{\text{Re}\,x^*}{\|x^*\|} \otimes x$ with $x^*x = 1$, we have $\|I - T\| = 2$.
		
	\end{enumerate}
\end{Theorem}

\begin{proof}
	
	The implication (i) $\iff$ (ii) $\iff$ (iv) was established in \cite{IK, P}. Hence, we prove the implications (iii) $\implies$ (ii) $\implies$ (vi) $\implies$ (iv), and (vi) $\implies$ (v) $\implies$ (iii). We additionally provide a different approach to the proof of (i) $\iff$ (ii). 
		
	(i) $\implies$ (ii): Let $\epsilon > 0$ and $S$ be a slice containing $x \in S_X$. If $x$ is a $\Delta$-point on $X$, then $\overline{conv}\Delta_{\epsilon}(x) \cap S \neq \emptyset$. From the fact that the set $B_X \setminus S$ is a closed convex set, we see that $\Delta_{\epsilon}(x) \not\subset B_X \setminus S$. Hence, $S \cap \Delta_{\epsilon}(x) \neq \emptyset$ and so there exists $y \in S$ such that $\|x- y\| \geq 2 - \epsilon$.

	
	(ii) $\implies$ (i): 
	Suppose that $x \notin \overline{conv}\Delta_{\epsilon}(x)$ for some $\epsilon > 0$. Notice that a singleton $\{x\}$ is convex as well as $\overline{conv}\Delta_{\epsilon}(x)$. Moreover, the set $\{x\}$ is compact. So by Hahn-Banach separation theorem, there exist $x^* \in S_{X^*}$ and $\alpha > 0$ such that the slice $S = S(x^*, \alpha)$ contains $x$ and $S \cap \overline{conv}\Delta_{\epsilon}(x) = \emptyset$. Hence, we see that $S \cap \Delta_{\epsilon}(x) = \emptyset$, and so $\|x - y\| < 2 - \epsilon$ for every $y \in S$.

	(iii) $\implies$ (ii): Consider a slice $S(x^* ,\delta)$ containing $x$ and $\epsilon > 0$. Then there exist $\delta_1> 0$ such that $\frac{\sqrt{2\delta_1}}{1 - \delta_1} < \frac{\epsilon}{4}$ and a bounded linear functional $y^* \in S_{X^*}$ such that $x \in S(y^*, \delta_1) \subset S(x^*, \delta)$.
	
	Consider a rank-one projection $P: y \mapsto y^*y\frac{x}{y^*x}$. Then by the assumption (iii), for every $\beta < \frac{\epsilon}{2}$, there exists $y \in S_X$ such that
	\begin{equation}\label{eq:control}
	\| y - Py\| = \left\|y - y^*y\frac{x}{y^*x}\right\| = \left \|\gamma y - \text{Re}\,y^*(\gamma y)\frac{x}{y^*x}\right\| > 2 - \beta,
	\end{equation}
	where $\gamma \in S_\mathbb{C}$ such that $|y^*y| = \gamma y^*y = y^*(\gamma y) = \text{Re}\,y^*(\gamma y)$. 
	
	Moreover, for $\beta$ small enough we also see that $\gamma y \in S(y^*, \delta_1)$. Let $\tilde{y} = \gamma y$. Then by (\ref{eq:control}), we obtain 
	\begin{eqnarray*}
	\|\tilde{y}- x\| &=& \left\|\gamma y - \text{Re}\,y^*(\gamma y)\frac{x}{y^*x} + \text{Re}\,y^*(\gamma y)\frac{x}{y^*x} - x\right\|\\
	&\geq& \left\|\gamma y - \text{Re}\,y^*(\gamma y)\frac{x}{y^*x}\right\| - \left\|x - \text{Re}\,y^*(\gamma y)\frac{x}{y^*x} \right\|\\
	&>& 2 - \beta - \left|1 - \frac{\text{Re}\,y^*(\gamma y)}{y^*x}\right|.   
	\end{eqnarray*}
	Since $|y^*x| \geq \text{Re}\,y^*x > 1 - \delta_1$, we can see that
	\[
	\left|1 - \frac{\text{Re}\,y^*(\gamma y)}{y^*x}\right| = \frac{|y^*x - \text{Re}\,y^*(\gamma y)|}{|y^*x|}\leq \frac{|1 - y^*x| + (1 - \text{Re}\,y^*(\gamma y))}{1 - \delta_1}
	\]
	and 
	\[
	(\text{Im}\,y^*x)^2 = |y^*x|^2 - (\text{Re}\,y^*x)^2 < 1 - (1 - \delta_1)^2 = 2\delta_1 - \delta_1^2.
	\]
	Hence, we have
	\begin{eqnarray*}
	|1 - y^*x| = \sqrt{(1 - \text{Re}\,y^*x)^2 + (\text{Im}\,y^*x)^2} &<& \sqrt{(1 - \text{Re}\,y^*x)^2 + 2\delta_1 - \delta_1^2}\\
	&<& \sqrt{\delta_1^2 + 2\delta_1 - \delta_1^2} = \sqrt{2\delta_1}.
	\end{eqnarray*}
	These consequently show that
	\[
	\|\tilde{y} - x\| > 2 - \beta - \frac{2\sqrt{2\delta_1}}{1 - \delta_1} > 2 - \epsilon.
	\]

	(ii) $\implies$ (vi) : Consider a rank-one linear operator $T = \frac{\text{Re}\,x^*}{\|x^*\|} \otimes x$ , where $x^* \in X^*$ satisfies $x^*x = 1$. Let $0 < \epsilon < \frac{1}{10}$, $\alpha \in \left(0, \frac{\epsilon}{(1 - \epsilon)\|x^*\|}\right)$, and $S = \left\{y \in B_X : \frac{\text{Re}\,x^*y}{\|x^*\|} > \frac{1}{\|x^*\|} - \alpha\right\}$. Then by the assumption, there exists $y \in S$ such that $\|x - y\| \geq 2 - \epsilon$. Moreover, from the fact that $\frac{y}{\|y\|} \in S$ and $\frac{1}{\|y\|} \leq \frac{1}{1 - \epsilon}$, we have 
	\[
	\left\|x - \frac{y}{\|y\|}\right\| > \|x - y\| - \left(\frac{1}{\|y\|} - 1\right) \geq 2 - \epsilon - \frac{\epsilon}{1 - \epsilon} > 2 - \frac{2\epsilon}{1- \epsilon}.
	\]
	
	Now, define a convex function $g(\lambda) : = \left\|\lambda \frac{x}{\|x^*\|} - (1 - \lambda)\frac{y}{\|y\|}\right\|$ on $[0,1]$. Notice that $g(0) = 1$, $g(1) = \frac{1}{\|x^*\|}$, and 
	\begin{align*}
	g\left(\frac{1}{2}\right) &= \frac{1}{2} \left\| \frac{x}{\|x^*\|} - \frac{y}{\|y\|}\right\|= \frac{1}{2} \left\| \frac{x}{\|x^*\|} - x + x - \frac{y}{\|y\|}\right\|\\
	&	\ge \frac{1}{2} \left( \left\|x-\frac{y}{\|y\|}\right \| - \left(1-\frac{1}{\|x^*\|}\right) \right)  > \frac{1}{2} + \frac{1}{2\|x^*\|}-\frac{\epsilon}{1- \epsilon}.
	\end{align*} Since $g(\lambda)$ is convex, we see that, for every $\lambda \in [0, \frac{1}{2})$, \[\frac{g(1) - g(1/2)}{1/2} \geq \frac{g(1/2) - g(\lambda)}{1/2- \lambda} > \frac{g(1/2) - g(\lambda)}{1/2}.\]  Hence,
	\[
	\frac{1}{2\|x^*\|} - \frac{1}{2} + \frac{\epsilon}{1- \epsilon} > g(1) - g\left(\frac{1}{2}\right) > g\left(\frac{1}{2}\right) - g(\lambda) > \frac{1}{2} + \frac{1}{2\|x^*\|} - \frac{\epsilon}{1- \epsilon} - g(\lambda),
	\]
	and so $g(\lambda) > 1 - \frac{2\epsilon}{(1 - \epsilon)}$ for every $\lambda \in [0, \frac{1}{2})$. From the continuity of $g(\lambda)$,  $g(\lambda) \geq 1 - \frac{2\epsilon}{1 - \epsilon}$ for every $\lambda \in [0, \frac{1}{2}]$. Also, for $\lambda \in (\frac{1}{2}, 1]$, we have 
	\[
	\frac{g(0) - g(1/2)}{\lambda - 1/2} > \frac{g(0) - g(1/2)}{1/2} \geq \frac{g(1/2) - g(\lambda)}{\lambda - 1/2}.\] 
	Hence we see that 
	\[
	1 - 1 + \frac{2\epsilon}{1 - \epsilon} \geq g(0) - g\left(\frac{1}{2}\right) > g\left(\frac{1}{2}\right) - g(\lambda) \geq 1 - \frac{2\epsilon}{1 - \epsilon} - g(\lambda),
	\]
	and so $g(\lambda) > 1 - \frac{4\epsilon}{1 - \epsilon}$ for every $\lambda \in (\frac{1}{2}, 1]$. 

	Let us consider two cases. First, assume that $\frac{\text{Re}\,x^*\frac{y}{\|y\|}}{1 + \text{Re}\,x^*\frac{y}{\|y\|}} \leq \frac{1}{2}$. Since $g(\lambda) \geq 1 - \frac{2\epsilon}{(1 - \epsilon)}$ for every $\lambda \in [0,\frac{1}{2}]$, we obtain
	\begin{eqnarray*}
		\left\|(I - T)\left(\frac{y}{\|y\|}\right)\right\| &=& \left(1 + \text{Re}\,x^*\frac{y}{\|y\|}\right) \cdot g\left(\frac{\text{Re}\,x^*\frac{y}{\|y\|}}{1 + \text{Re}\,x^*\frac{y}{\|y\|}}\right)\\
		&\geq&\left(1 + \text{Re}\,x^*\frac{y}{\|y\|}\right) \left(1 - \frac{2\epsilon}{1 - \epsilon}\right)\\
		&>& (2 - \alpha\|x^*\|)\left(1 - \frac{2\epsilon}{1 - \epsilon}\right)\\
		&>& \left(2 - \frac{\epsilon}{1 - \epsilon}\right) \left(1 - \frac{2\epsilon}{1 - \epsilon}\right)\\
		&=& 2 - \frac{\epsilon}{1 - \epsilon} - \frac{4\epsilon}{1 - \epsilon} + 2\left(\frac{\epsilon}{1 - \epsilon}\right)^2\\
		&>& 2 - \frac{5\epsilon}{1 - \epsilon}.
	\end{eqnarray*}
    
	Now, assume that $\frac{\text{Re}\,x^*\frac{y}{\|y\|}}{1 + \text{Re}\,x^*\frac{y}{\|y\|}} > \frac{1}{2}$. Since $g(\lambda) > 1 - \frac{4\epsilon}{(1 - \epsilon)}$ for every $\lambda \in (\frac{1}{2},1]$, we have
	\begin{eqnarray*}
		\left\|(I - T)\left(\frac{y}{\|y\|}\right)\right\| &=& \left(1 + \text{Re}\,x^*\frac{y}{\|y\|}\right) \cdot g\left(\frac{\text{Re}\,x^*\frac{y}{\|y\|}}{1 + \text{Re}\,x^*\frac{y}{\|y\|}}\right)\\
		&>& \left(1 + \text{Re}\,x^*\frac{y}{\|y\|}\right)\left(1 - \frac{4\epsilon}{1 - \epsilon}\right)\\
		&>& (2 - \alpha\|x^*\|)\left(1 - \frac{4\epsilon}{1 - \epsilon}\right)\\
		&>& \left(2 - \frac{\epsilon}{1 - \epsilon}\right)\left(1 - \frac{4\epsilon}{1 - \epsilon}\right)\\
		&=& 2 - \frac{\epsilon}{1-\epsilon} - \frac{8\epsilon}{1 - \epsilon} + 4\left(\frac{\epsilon}{1-\epsilon}\right)^2\\
		&>& 2 - \frac{9\epsilon}{(1 - \epsilon)}.
	\end{eqnarray*}
	Therefore, from the fact that $\epsilon > 0$ is arbitrary in all cases, we have (ii) $\implies$ (vi).

(vi) $\implies$ (iv): Let $T = \frac{\text{Re}\,x^*}{\|x^*\|} \otimes x$, where $x^* \in X^*$ satisfies $x^*x = 1$. Define a function $\varphi(\lambda) = \|I - \lambda T\|$, where $\lambda \in [0, \infty)$. It is easy to show that $\varphi$ is a convex function on $[0, \infty)$. Also, $\varphi(0) = 1$ and $\varphi(1) = \|I - T\| =2$ from the assumption. Furthermore,
	\[
	0 < \frac{\varphi(1) - \varphi(0)}{1-0} \leq \frac{\varphi(s) - \varphi(1)}{s-1}\,\,\, \text{for all} \,\,\, s \geq 1.
	\]
	So, for every $s \geq 1$, we see that $\varphi(s) > \varphi(1) = 2$. In particular, if $s = \|x^*\| \geq 1$, we obtain $\varphi(\|x^*\|) = \|I - \|x^*\|\cdot T\| = \|I - P\| \geq 2$. Therefore, the implication (vi) $\implies$ (iv) holds. 
	
	(vi) $\implies$ (v) : Let $T = \frac{x^*}{\|x^*\|} \otimes x$, where $x^* \in X^*$ satisfies $x^*x = 1$. By the assumption, for every $\epsilon > 0$, there exists $y \in S_X$ such that $\left\|y - \frac{\text{Re}\,x^*y}{\|x^*\|}x\right\| \geq 2 - \frac{\epsilon^2}{2}$. Notice that $\left| \frac{\text{Re}\,x^*y}{\|x^*\|}\right| \geq 1 - \frac{\epsilon^2}{2}$ and so $\left|\frac{\text{Im}\, x^*y}{\|x^*\|}\right| \leq \epsilon$. 
	
	Then we have
	\begin{eqnarray*}
		\left\|y - \frac{x^*y}{\|x^*\|}x\right\| \geq \left\|y - \frac{\text{Re}\, x^*y}{\|x^*\|}x\right\|  - \left\|\frac{x^*y}{\|x^*\|}x - \frac{\text{Re}\, x^*y}{\|x^*\|}x\right\| &\geq& 2 - \frac{\epsilon^2}{2} - \left|\frac{\text{Im}\, x^*y}{\|x^*\|}\right|\\
		&\geq& 2 - \frac{\epsilon^2}{2} - \epsilon. 	
	\end{eqnarray*}
	This shows (vi) $\implies$ (v). The proof of (v) $\implies$ (iii) is identical to (vi) $\implies$ (iv). 
\end{proof}

Hence we can verify the relationship between the $\Delta$-points and the DLD2P as well as the spaces with bad projections \cite{IK} for complex Banach spaces.

\begin{corollary}
	Let $X$ be a complex Banach space. The following statements are equivalent:
	\begin{enumerate}[\rm(i)]
		\item The space $X$ has the DLD2P.
		\item Every point on the unit sphere $S_X$ is a $\Delta$-point.
		\item Every rank-one projection $P$ has $\|I - P\| \geq 2$, i.e. $X$ is a space with bad projections.
		\item Every rank-one real-projection $P$ has $\|I - P\| \geq 2$.
	\end{enumerate}
\end{corollary} 

The following statement about Daugavet points is a compilation of well-known results, but we include them for completeness.

\begin{Theorem}
	Let $X$ be a complex Banach space and $x \in S_X$. Then the following statements are equivalent.
	\begin{enumerate}[\rm(i)]
		\item $x$ is a Daugavet point.
		\item For every slice $S$ of $B_X$ and $\epsilon > 0$, there exists $y \in S$ such that $\|x - y\| \geq 2 - \epsilon$.
		\item Every rank-one operator $T = x^* \otimes x$ satisfies $\|I - T\| = 1 + \|T\|$.
		\item Every rank-one operator $T = x^* \otimes x$ with norm-one satisfies $\|I - T\| = 2$.
		\item Every rank-one real-linear operator $T = \text{Re}\,x^* \otimes x$ satisfies $\|I - T\| = 1 + \|T\|$.
		\item Every rank-one real-linear operator $T = \text{Re}\, x^* \otimes x$ with norm-one satisfies $\|I - T\| = 2$.  
	\end{enumerate}
\end{Theorem}

\begin{proof}
	One may see the proof for the real case in \cite[Proposition 1.4.5]{P}.
	
	(i) $\implies$ (ii):  Let $S$ be a slice of $B_X$ and $\epsilon > 0$. If $x$ is a Daugavet point, then $S \subset \overline{\text{conv}} \Delta_{\epsilon}(x) = B_X$. By using a similar reasoning as (i) $\implies$ (ii) in Theorem \ref{prop:deltaequiv}, we see that $S \cap \Delta_{\epsilon}(x) \neq \emptyset$. Therefore, there exists $y \in S$ such that $\|x - y\| \geq 2 - \epsilon$.

	
	(ii) $\implies$ (i): Suppose that $B_X \neq \overline{conv}\Delta_{\epsilon}(x)$ for some $\epsilon > 0$. A singleton $\{x\}$ is compact and convex, and the set $\overline{conv}\Delta_{\epsilon}(x)$ is also convex. So by a similar argument as (ii) $\implies$ (i) in Theorem \ref{prop:deltaequiv}, there exists a slice $S$ such that $S \cap \Delta_{\epsilon}(x) = \emptyset$. Therefore, for every $y \in S$, we have $\|x - y\| < 2 - \epsilon$.

			
	(iii) $\implies$ (iv) is clear, and (iv) $\implies$ (iii) comes from the fact that for Daugavet property considering rank-one operators with norm-one is enough \cite{W}. The equivalence (ii) $\iff$ (v) $\iff$ (vi) is a well-known result in \cite{AHLP}.
	
	(iv) $\implies$ (ii): Let $x^* \in S_{X^*}$, $\epsilon > 0$, and $S = \{y \in B_X: \text{Re}\,x^*y > 1 - \frac{\epsilon}{2}\}$. Let $T \in L(X)$ be the rank-one operator with norm one defined by $T(y) = x^*y\cdot x$. By the assumption (iv), there exists $y \in S_X$ such that $\|y - Ty\| = \|y - x^*y\cdot x\| > 2 -  \frac{\epsilon}{2}$. Notice that $|x^*y| = \text{Re}\,x^*(\gamma y)$ for some $\gamma \in S_{\mathbb{C}}$, and so
	\[
	1 + \text{Re}\,x^*(\gamma y) \geq \|y - x^*y \cdot x\| > 2 - \frac{\epsilon}{2}.
	\]
	Hence $\gamma y \in S$. Let $\tilde{y} = \gamma y$. Then we have
	\[
	\|\tilde{y} - x\| \geq \|\tilde{y} - \text{Re}\,x^*(\gamma y) \cdot x\| - 1 + \text{Re}\,x^*(\gamma y) > 2 - \epsilon.
	\]
	Therefore, we see that (ii) holds.
	
	(ii) $\implies$ (iv): Consider $T = x^* \otimes x$ where $\|x^*\| = 1$. Let $\epsilon > 0$ and $S = \{y \in B_X: \text{Re}\,x^*y > 1 - \frac{\epsilon}{2}\}$ be a slice of $B_X$. By the assumption (ii), there exists $y \in S$ such that $\|x - y\| > 2 - \frac{\epsilon}{2}$. Notice that $1 \geq |x^*y|^2 = (\text{Re}\,x^*y)^2 + (\text{Im}\,x^*y)^2$. Hence we have  
	\[
	|1- x^*y| = \sqrt{(1 - \text{Re}\,x^*y)^2 + (\text{Im}\,x^*y)^2} < \sqrt{\frac{\epsilon^2}{4} + 1 - \left(1 - \frac{\epsilon}{2}\right)^2} = \sqrt{\epsilon}.  
	\]
	Moreover,
	\begin{eqnarray*}
	\|(I - T)y\| = \|y - x^*yx\| \geq \|y - x\| - \|x - x^*y \cdot x\| &>& 2 - \frac{\epsilon}{2} - |1 - x^*y|\\
	&>& 2 - \frac{\epsilon}{2}  - \sqrt{\epsilon} > 2 - \epsilon.  
	\end{eqnarray*}
	Since $\epsilon > 0$ is arbitrary, we obtain $\|I - T\| \geq 2$. Then (iv) holds immediately from the fact that $\|I - T\| \leq 1 + \|T\| = 2$.
\end{proof}

\begin{Corollary}
	Let $X$ be a complex Banach space. Then the following statements are equivalent:
	\begin{enumerate}[\rm(i)]
		\item The space $X$ has the Daugavet property.
		\item Every point on the unit sphere $S_X$ is a Daugavet point.
	\end{enumerate}
\end{Corollary}

Now, we make a similar observation on rank-one, norm-one projections.

\begin{Proposition}\label{prop:propD}
	Let $X$ be a complex Banach space and suppose that there are $x^*\in X^*$ and $x\in X$ such that  $x^*x=1=\|x^*\|=\|x\|$. Then the rank-one projection $P = x^* \otimes x$ on $X$ satisfies $\|I - P\| = 2$ if and only if the rank-one real-projection $P' = \text{Re}\,x^*\otimes x$ on $X$ satisfies $\|I - P'\| = 2$.
\end{Proposition}

\begin{proof}
	Let $P' = \text{Re}\,x^* \otimes x$. Then for every $\epsilon >0$, there exists $y \in S_X$ such that $\|y - x^*y\cdot x\| \geq 2 - \epsilon$. Now take $\gamma \in S_{\mathbb{C}}$ such that $|x^*y| = \gamma x^*y$. This implies that $x^*(\gamma y) = \text{Re}\, x^*(\gamma y)$. Hence we see that 
	\[
	\|I - P\| \geq \|\gamma y  - \text{Re}\,x^*(\gamma y)\cdot x\| = \|\gamma( y  - x^*y\cdot x)\| = \|y  - x^*y\cdot x\| \geq 2 -\epsilon.  
	\]  
	Since $\epsilon > 0$ is arbitrary, we obtain $\|I - P\| = 2$.
	
	Let $P = x^* \otimes x$. Then we see that for every $\epsilon > 0$, there exists $y \in S_X$ such that $\|y - \text{Re}\,x^*y\cdot x\| \geq 2 - \frac{\epsilon}{2}$. Then we have $|\text{Re}\,x^*y| \geq 1- \frac{\epsilon}{2}$. Notice that
	\[
	(\text{Im}\,x^*y)^2 = |x^*y|^2 - (\text{Re}\,x^*y)^2 \leq 1 - \left(1 - \frac{\epsilon}{2}\right)^2 < \epsilon.
	\]
	Hence, we obtain
	\[
	\|I - P\| \geq \|y - x^*y\cdot x\| \geq \|y - \text{Re}\,x^*y \cdot x\| - |\text{Im}\,x^*y| > 2 - 2\epsilon.
	\]
	Therefore, since $\epsilon > 0$ is arbitrary, we obtain $\|I - P\| = 2$.
\end{proof}

\section{Alternative convexity or smoothness, nonsquareness, and the Daugavet property}

In this section, we study the relationship between alternative convexity or smoothness, nonsquareness, and the Daugavet property. First, we recall various nonsquareness properties, in the sense of James \cite{J,WSL}. The uniform nonsquareness has been examined on both real and complex Banach spaces via Jordan-von Neumann constants \cite{KMT}. 

\begin{definition} 
	\begin{enumerate}[\rm(i)]
		\item A Banach space $X$ is uniformly nonsquare (UNSQ) if there exists $\delta > 0$ such that for every $x, y \in S_X$, $\min\{\|x\pm y\|\} \leq 2 - \delta$.
		\item A Banach space $X$ is locally uniformly nonsquare (LUNSQ) if for every $x \in S_X$, there exists $\delta > 0$ such that $\min\{\|x \pm y\|\} \leq 2 - \delta$ for every $y \in S_X$.
		\item A Banach space $X$ is nonsquare (NSQ) if for every $x, y \in S_X$, $\min\{\|x \pm y\| \} < 2$.
	\end{enumerate}
\end{definition}
\noindent Here we call each point $x \in S_X$ in (ii) a {\it locally uniformly nonsquare point (or uniformly non-$\ell_1^2$ point)}. We have the following implication for these classes:
\[
\text{UNSQ}\,\,\, \implies\,\,\, \text{LUNSQ}\,\,\, \implies\,\,\, \text{NSQ}. 
\]

It has been recently shown that UNSQ real Banach spaces do not have $\Delta$-points \cite{ALMP} at all. We extend this result for the class of LUNSQ Banach spaces on the complex scalar field. Let us start with an improvement of Theorem \ref{prop:deltaequiv}. For the proof on the real case, we refer to \cite[Lemma 2.1]{JR}.

\begin{Lemma}\label{lem:delta}
	Let $X$ be a complex Banach space and $x\in S_X$ be a $\Delta$-point. For every $\epsilon > 0$, $\frac{\alpha}{1 - \alpha} < \epsilon$, and every slice $S = S(x^*, \alpha)$ containing $x$, there exists a slice $S(z^*, \alpha_1)$ of $B_X$ such that $S(z^*, \alpha_1) \subset S(x^*, \alpha)$ and $\|x - y\| > 2 - \epsilon$ for all $y \in S(z^*, \alpha_1)$. 
\end{Lemma} 

\begin{proof}
	 First choose $\eta > 0$ such that $\eta < \min\left\{1 - \frac{1 - \alpha}{\text{Re}\,x^*(x)}, \epsilon - \frac{\alpha}{1 - \alpha}\right\}$. Since $x \in S_X$ is a $\Delta$-point, for the projection $P(y) = \frac{\text{Re}\,x^*y}{\text{Re}\,x^*x}\cdot x$ we have $\|I - P\| \geq 2$ by Theorem \ref{prop:deltaequiv}. Then there exists $y^* \in S_{X^*}$ such that $\|y^* - P^*y^*\| \geq 2 - \eta$. Now define $z^* = \frac{P^*y^* - y^*}{\|P^*y^* - y^*\|} \in S_{X^*}$ and $\alpha_1 = 1 - \frac{2 - \eta}{\|P^*y^* - y^*\|}$ where $P^*y^* = \frac{y^*x}{\text{Re}\,x^*x}\cdot \text{Re}\, x^*$. For every $y \in S(z^*, \alpha_1)$ notice that
	\[
	\frac{\text{Re}\,x^*y}{\text{Re}\,x^*x}\cdot \text{Re}\,y^*x - \text{Re}\,y^*y = \text{Re}\,z^*y\cdot \|P^*y^* - y^*\| > 2 - \eta. 
	\]
	Hence we see that $\frac{\text{Re}\,x^*y}{\text{Re}\,x^*x}\cdot \text{Re}\,y^*x > 1 - \eta$. Since $\text{Re}\, y^*x$ cannot be zero, without loss of generality, assume that $\text{Re}\, y^*x > 0$. Then we have $\text{Re}\, x^*y > (1 - \eta)\cdot\text{Re}\,x^*x > 1 - \alpha$, which shows that $S(z^*, \alpha_1) \subset S(x^*, \alpha)$. Furthermore, notice that $\text{Re}\,x^*y \leq |x^*y| \leq 1$, and so
	\[
	\left\|\frac{x}{\text{Re}\,x^*x} - y\right\|\geq \frac{\text{Re}\,y^*x}{\text{Re}\,x^*x} - \text{Re}\,y^*y \geq 
	\frac{\text{Re}\,x^*y}{\text{Re}\,x^*x}\cdot\text{Re}\,y^*x - \text{Re}\,y^*y > 2 - \eta. 	
	\] 
	Therefore, we obtain
	\[
	\|x - y\| \geq \left\|\frac{x}{\text{Re}\,x^*x} - y\right\| - \left(\frac{1}{\text{Re}\,x^*x}-1\right) > (2 - \eta) - \left(\frac{\alpha}{1 - \alpha}\right) > 2 - \epsilon.
	\]
\end{proof}

Applying Lemma \ref{lem:subslice} again, we can also show the converse of the previous statement. The same proof in \cite[Lemma 2.2]{JR} transfers to complex Banach spaces. 

\begin{Corollary}
	Let $X$ be a complex Banach space. Then $x \in S_X$ is a $\Delta$-point if and only if for every $\epsilon > 0$ and every slice $S = S(x^*, \alpha)$ containing $x$, there exists a slice $S(z^*, \alpha_1)$ such that $S(z^*, \alpha_1) \subset S(x^*, \alpha)$ and $\|x - y\| > 2 - \epsilon$ for all $y \in S(z^*, \alpha_1)$.
\end{Corollary} 

As a consequence, we obtain the relationship between the locally uniformly nonsquare points and the $\Delta$-points on both real and complex Banach spaces. The real case was shown in \cite[Theorem 4.2]{KLT}.

\begin{Proposition}\label{prop:nod}
	Let $X$ be a real or complex Banach space. A locally uniformly nonsquare point $x \in S_X$ is not a $\Delta$-point of $X$. 
\end{Proposition}

\begin{proof}
	We show that a $\Delta$-point $x \in S_X$ cannot be a locally uniformly nonsquare point. Let $\epsilon > 0$ and $\eta \in (0, \frac{\epsilon}{2})$. By Lemma \ref{lem:delta}, for every $\alpha > 0$ where $\frac{\alpha}{1 - \alpha} < \eta$ and every slice $S = S(x^*, \alpha)$ containing $x$, there exists a slice $S(z^*, \alpha_1) \subset S$ such that $\|x - y\| > 2 - \eta > 2 - \epsilon$ for all $y \in S(z^*, \alpha_1)$. In particular, we have $\frac{\text{Re}\,z^*y}{\|y\|} > \frac{1-\alpha_1}{\|y\|} \geq 1 - \alpha_1$ for every $y \in S(z^*, \alpha_1)$. Hence $y' = \frac{y}{\|y\|} \in S(z^*, \alpha_1)$ and $\|x - y'\| > 2 - \epsilon$ for $y \in S(z^*, \alpha_1)$.   
	
	Moreover, by the fact that $\alpha < \frac{\alpha}{1 - \alpha} < \eta$ and $x, y \in S$, we have $\|x + y'\| \geq \text{Re}\,x^*x + \text{Re}\,x^*y' > 2 - 2\alpha > 2 - 2\eta > 2 - \epsilon$. Thus, for every $\epsilon > 0$ there exists $y' \in S_X$ such that $\min\{\|x + y'\|, \|x - y'\|\} > 2 - \epsilon$. This shows that $x \in S_X$ is not a locally uniformly nonsquare point.
\end{proof}

\begin{Corollary}\label{cor:nodelta}
	Let $X$ be a real or complex Banach space. If $X$ is LUNSQ, then the space does not admit $\Delta$-points. As a consequence, every LUNSQ space does not satisfy the Daugavet property, DD2P, and DLD2P. 
\end{Corollary}

A Banach space $X$ is said to have the {\it anti-Daugavet property} for a class of operators $\mathcal{M}$ if the following equivalence holds:
\[
\|I + T\| = 1 + \|T\| \iff \|T\| \in \sigma(T),
\]
where $\sigma(T)$ is the spectrum of $T \in \mathcal{M}$. If $\mathcal{M} = L(X)$, we simply say the space $X$ satisfies the anti-Daugavet property. We mention that the only if part always holds for any bounded linear operator. 

It is well-known that any uniformly rotund or uniformly smooth Banach spaces have the anti-Daugavet property. Moreover, this property is connected to the alternative convexity or smoothness properties that are introduced in \cite{KSSW}:

\begin{definition}\label{def:acs}
	\begin{enumerate}[\rm(i)]
		\item A Banach space $X$ is uniformly alternatively convex or smooth (uacs) if for all sequences $(x_n), (y_n) \subset S_X$, $(x_n^*) \subset S_{X^*}$, $\|x_n + y_n\|\rightarrow 2$, and $x_n^*(x_n) \rightarrow 1$ implies $x_n^*(y_n) \rightarrow 1$.
		\item A Banach space $X$ is strongly locally uniformly alternatively convex or smooth (sluacs) if for every $x \in S_X$, $(x_n) \subset S_X$, $(x^*_n) \subset S_{X^*}$, $\|x_n + x\|\rightarrow 2$, and $x_n^*(x_n) \rightarrow 1$ implies $x_n^*(x) \rightarrow 1$. 
		\item A Banach space $X$ is alternatively convex or smooth (acs) if for all $x,y \in S_X$, $x^* \in S_{X^*}$, $\|x + y\|= 2$, and $x^*(x) = 1$ implies $x^*(y) = 1$. 
	\end{enumerate}
\end{definition}
\noindent 
Any uniformly convex (resp. locally uniformly rotund, rotund) and uniformly smooth (resp. uniformly Gateaux-smooth, smooth) Banach spaces are known to be uacs (resp. sluacs, acs) \cite{Hard}.

Even though it is mentioned in \cite{KSSW} that alternative convexity or smoothness for complex Banach spaces can be defined in a similar fashion, any recent investigation on this property also assumes the scalar field to be $\mathbb{R}$. Hence, we provide equivalent definitions that only involve the real part of bounded linear functionals which enable us to consider complex Banach spaces.

\begin{Proposition}
	\begin{enumerate}[\rm(i)]
		\item A Banach space $X$ is uacs if and only if for all sequence $(x_n), (y_n) \subset S_X$, $(x_n^*) \subset S_{X^*}$, $\|x_n + y_n\| \rightarrow 2$, and $\text{Re}\,x_n^*x_n \rightarrow 1$ implies $\text{Re}\,x_n^*y_n \rightarrow 1$.
		\item A Banach space $X$ is sluacs if and only if for every $x \in S_X$, $(x_n) \subset S_X$, $(x_n^*) \subset S_{X^*}$, $\|x_n + x\| \rightarrow 2$ and $\text{Re}\,x_n^*x_n \rightarrow 1$ implies $\text{Re}\,x_n^*x \rightarrow 1$.
		\item A Banach space $X$ is acs if and only if for all $x,y \in S_X$ and $x^* \in S_{X^*}$, $\|x + y\|= 2$ and $\text{Re}\,x^*(x) = 1$ implies $\text{Re}\,x^*(y) = 1$. 
	\end{enumerate}
\end{Proposition}

\begin{proof}
We assume that $X$ is a complex Banach space. Since the proofs for (i) and (ii) are similar, we only prove (i).

Suppose that $(x_n), (y_n) \subset S_X$, $(x_n^*) \subset S_{X^*}$, $x_n^*x_n \rightarrow 1$ and $\|x_n + y_n\| \rightarrow 2$. Then for every $\epsilon > 0$, there exists $N_1 \in \mathbb{N}$ such that for every $n \geq N_1$, $|x_n^*x_n| > 1 - \frac{\epsilon}{2}$. We see that
\[
1 \geq (\text{Re}\,x_n^*x_n)^2  + (\text{Im}\,x_n^*x_n)^2  = |x_n^*x_n|^2 > \left(1 - \frac{\epsilon}{2}\right)^2 + (\text{Im}\,x_n^*x_n)^2.
\] 	
Hence $\text{Im}\,x_n^*x_n < \sqrt{\epsilon}$, which in turn implies that $\text{Re}\,x_n^*x_n \rightarrow 1$. Then by the assumption, we obtain $\text{Re}\,x_n^*y_n \rightarrow 1$. Thus, for every $\epsilon > 0$, there exists $N_2 \in \mathbb{N}$ such that $\text{Re}\,x_n^*y_n > 1 - \frac{\epsilon}{2}$. This implies that
\[
1 \geq |x_n^*y_n|^2 = (\text{Re}\,x_n^*y_n)^2 + (\text{Im}\,x_n^*y_n)^2 \geq \left(1 - \frac{\epsilon}{2}\right)^2 + (\text{Im}\, x_n^*y_n)^2,
\]
and so $\text{Im}\,x_n^*y_n < \sqrt{\epsilon}$. Therefore, we see that $x_n^* y_n = \text{Re}\,x_n^* y_n + i\text{Im}\,x_n^*y_n \rightarrow 1$. 

For (iii), if for $x, y \in S_X$ and $x^* \in S_{X^*}$ we have $\|x +y\| = 2$ and $x^*x = \text{Re}\,x^*x = 1$, then $\text{Re}\,x^*y = 1$ by the assumption. Hence, we see that
\[
1 \geq |x^*y|^2 = (\text{Re}\,x^*y)^2 + (\text{Im}\,x^*y)^2 = 1 + (\text{Im}\,x^*y)^2.
\]   
Therefore, $x^*y = \text{Re}\,x^*y = 1$. 
\end{proof}

Even though every uacs Banach spaces are UNSQ \cite{Hard, KSSW}, there have been no explicit description on the relationship between the sluacs and LUNSQ (resp. acs and NSQ) Banach spaces. As a matter of fact, a similar statement also holds for both sluacs and acs Banach spaces.

\begin{Proposition}\label{prop:acsnsq}
	\begin{enumerate}[\rm(i)]
		\item Every sluacs space is LUNSQ.
		\item Every acs space is NSQ.
	\end{enumerate}
\end{Proposition} 

\begin{proof}
	(i) Suppose that a sluacs space $X$ is not LUNSQ. Then there exists $x \in S_X$ such that for every $\delta > 0$, there exists $y \in S_X$ such that $\|x + y\| > 2 - \delta$ and $\|x - y\| > 2 - \delta$. So choose a sequence $(x_n)_{n=1}^{\infty} \subset S_X$ such that $\|x + x_n\| > 2 - \frac{1}{2^n}$ and $\|x - x_n\| > 2 - \frac{1}{2^n}$. In view of Hahn-Banach theorem, we can also find a sequence $(x_n^*)_{n=1}^{\infty} \subset S_{X^*}$ such that
	\[
	\text{Re}\,x_n^*x_n - \text{Re}\,x_n^*x = \|x - x_n\|\,\,\, \text{and} \,\,\, \text{Re}\,x_n^*x_n + \text{Re}\,x_n^*x = \|x + x_n\|.
	\]
	
	Then we see that
	\[
	2 - \frac{1}{2^n} < \text{Re}\,x_n^*x_n + \text{Re}\,x_n^*x \leq \|x\| + \text{Re}\,x_n^*x_n = 1 + \text{Re}\,x_n^*x_n, 
	\]
	and so $\text{Re}\,x_n^*x_n \rightarrow 1$ as $n \rightarrow \infty$. Since the space $X$ is assumed to be sluacs, $\text{Re}\,x_n^*x \rightarrow 1$ as $n \rightarrow \infty$. However, by repeating the same argument to $\|x - x_n\|$, we also obtain that $-\text{Re}\,x_n^*x \rightarrow 1$. This leads to a contradiction.
	
	(ii) Suppose that an acs space $X$ is not nonsquare. Then there exist $x, y \in S_X$ such that $\|x + y\| = \|x - y\| = 2$. Let $x^* \in S_{X^*}$ such that $\text{Re}\,x^*(x) + \text{Re}\,x^*(y) = \|x + y\|$. From the fact that 
	\[
	2 = \text{Re}\,x^*(x) + \text{Re}\,x^*(y) \leq  \text{Re}\,x^*(x)  + \|y\| = \text{Re}\,x^*(x) + 1,
	\]
	we have $\text{Re}\,x^*x = 1$. This also shows that $\text{Re}\,x^*y = 1$. However, since $\|x - y\| = 2$ and the space $X$ is acs, we have $-\text{Re}\,x^*x = 1$, which is a contradiction. Therefore, the space $X$ must be nonsquare.   
\end{proof}

We mention that any locally uniformly rotund (LUR) Banach space does not have $\Delta$-points. As a matter of fact, we can show further that every sluacs Banach space does not have $\Delta$-points based on our observation. 

\begin{corollary}
	Let $X$ be a Banach space. Every sluacs Banach space does not contain $\Delta$-points.   
\end{corollary}

\begin{proof}
	Every sluacs Banach space is LUNSQ by Proposition \ref{prop:acsnsq}.(i). Then by Proposition \ref{prop:nod}, we see that the space does not contain $\Delta$-points.
\end{proof}

We mention that the analogous statement for NSQ Banach spaces cannot be obtained because there exists a rotund Banach space with the DLD2P \cite[Theorem 2.12]{AHNTT}. This fact also leads us to ask whether rotund or NSQ Banach spaces can have the Daugavet property, which has been a long standing open problem for decades. While there is a rotund normed space (not complete) with the Daugavet property \cite{KMMP}, the existence has not been verified for Banach spaces.

\section{Remarks on the Daugavet property of $A(K, X)$}

Let $K$ be a compact Hausdorff space. The space $C(K)$ is the set of all continuous functions over $K$ endowed with the supremum norm $\|\cdot\|_{\infty}$. A {\it uniform algebra} $A$ is a closed subalgebra of $C(K)$ that separates points and contains constant functions. For a compact subset $K \subset \mathbb{C}$, the space $P(K)$ (resp. $R(K)$) of continuous functions that can be approximated uniformly on $K$ by polynomials in $z$ (resp. by rational functions with poles off $K$) and the space $A(K)$ of continuous functions that are analytic on the interior of $K$ are well-known examples of uniform algebras. When $K = \overline{\mathbb{D}}$, the corresponding uniform algebra $A(K) = A(\overline{\mathbb{D}})$ is the disk algebra. We refer to \cite{D, L} for more details on uniform algebras. 

For a complex Banach space $X$, let $C(K, X)$ be the set of all vector-valued continuous functions over $K$ equipped with the supremum norm. We recall the definition of the vector-valued function space $A(K, X)$. 

\begin{Definition}
	Let $K$ be a compact Hausdorff space and $X$ be a Banach space. The space $A(K,X)$ is called {\it a function space over the base algebra $A$} if it is a subspace of $C(K,X)$ that satisfies:
	\begin{enumerate}[\rm(i)]
		\item The base algebra $A := \{x^*\circ f: x^* \in X^*, f \in A(K, X)\}$ is a uniform algebra.
		\item $A \otimes X  = \{f \otimes x: f \in A \,\,\, \text{and} \,\,\, x \in X\}
		\subset A(K, X)$.
		\item For every $g \in A$ and every $f \in A(K, X)$, we have $g\cdot f \in A(K, X)$.
	\end{enumerate}
\end{Definition}

\noindent If $X = \mathbb{F}$, then the space $A(K, X)$ becomes the uniform algebra $A$ on a compact Hausdorff space $K$. It is clear that $C(K, X)$ is a function space over a base algebra $C(K)$. As a nontrivial example, for given Banach spaces $X$ and $Y$,  let $A_{w^*u}(B_{X^*}, Y)$  be the space of all weak$^*$-to-norm uniformly continuous functions on the closed unit ball $B_{X^*}$ that are holomorphic on the interior of $B_{X^*}$. It is a closed subspace of $C(B_{X^*}; Y)$, where $B_{X^*}$ is the weak$^*$-compact set. Then $A_{w^*u}(B_{X^*}; Y)$ is a function space over base algebra $A_{w^*u}(B_{X^*})$.

A subset $L \subset K$ is said to be a {\it boundary} for $A$ if for every $f \in A$ there exists $t \in L$ such that $f(t) = \|f\|_\infty$. The smallest closed boundary for $A$ is called the {\it Shilov boundary} denoted by $\Gamma$. A point $x \in K$ is a {\it strong boundary point} for a uniform algebra $A$ if for every open subset $U \subset K$ containing $x$, there exists $f \in A$ such that $\|f\|_{\infty} = |f(x)| = 1$ and $\sup_{t \in K \setminus U} |f(t)| < 1$. For a compact Hausdorff space $K$, the set of all strong boundary points on $A$ coincides with the Choquet boundary $\Gamma_0$, that is, the set of all extreme points on the set $K_A = \{\lambda \in A^*: \|\lambda\| = \lambda(1_A) = 1\}$ \cite[Theorem 4.3.5]{D}, where $1_A$ is the unit of $A$. Moreover, the closure of $\Gamma_0$ is $\Gamma$ in this case \cite[Corollary 4.3.7.a]{D}. For instance, the Shilov boundary of the disk algebra $A(\overline{\mathbb{D}})$ is the unit circle $\partial\overline{\mathbb{D}}$.

To study various geometrical properties of $A(K, X)$ and Bishop-Phelps-Bollob\'as property for Asplund operators which range space is a uniform algebra, a Urysohn-type lemma has played an important role. Here we use a stronger version of the lemma provided in \cite{CGK}. 

\begin{Lemma}\cite[Lemma 3.10]{CJT}\label{lem:urysohnAKX}
	Let $K$ be a compact Hausdorff space. If $t_0$ is a strong boundary point for a uniform algebra $A \subset C(K)$, then for every open subset $U \subset K$ containing $t_0$ and $\epsilon > 0$, there exists $\phi = \phi_U \in A$ such that $\phi(t_0) = \|\phi\|_{\infty} = 1$, $\sup_{K \setminus U}|\phi(t)| < \epsilon$ and
	\[
	|\phi(t)| + (1 - \epsilon)|1 - \phi(t)| \leq 1
	\]
	for every $t \in K$.
\end{Lemma}
	We can also construct a Urysohn-type function at an isolated point in the Shilov boundary.
	
	\begin{Lemma}\cite[Lemma 2.5]{LT}\label{lem:auxiso}
		Let $A$ be a uniform algebra on a compact Hausdorff space $K$ and let $t_0$ be an isolated point of the Shilov boundary $\Gamma$ of $A$. Then there exists a function $\phi \in A$ such that $\phi(t_0) = \|\phi\|_{\infty} = 1$ and $\phi(t) = 0$ for $t \in \Gamma \setminus \{t_0\}$. 
	\end{Lemma}

We recall the following lemma that will be useful for later.
\begin{Lemma}\cite[Lemma 2.4]{LT}\label{lem:AKXisom}
	Let $X$ be a Banach space. Suppose that $L$ is a closed boundary for $A$. The space of restrictions of elements of $A(K, X)$ to $L$ is denoted by $A(L, X)$ and the restrictions of elements of $A$ to $L$ is denoted by $A(L)$. Then $A(L, X)$ is isometrically isomorphic to $A(K, X)$.
\end{Lemma}

The next statement is in the proof of the case (iii) for \cite[Theorem 4.2]{LT}, but we state it explicitly here. 

\begin{Lemma} \label{lem:akxdecomp}
	Let $K$ be a compact Hausdorff space and $\Gamma$ be the Shilov boundary of the base algebra for the space $A(K, X)$. Suppose that $\Gamma$ has an isolated point $t_0$. Then, $A(K, X)$ is isometrically isomorphic to $X \oplus_{\infty} Y$ where $Y$ is $A(\Gamma, X)$ restricted to $\Gamma \setminus \{t_0\}$.  
\end{Lemma}

\begin{proof}
	Let $t_0 \in \Gamma$ be an isolated point. By Lemma \ref{lem:auxiso}, there exists $\phi \in A$ such that $\phi(t_0) = \|\phi\|_{\infty} = 1$ and $\phi(t) = 0$ for $t \in \Gamma \setminus \{t_0\}$. Let $\tilde{\Gamma} = \Gamma \setminus \{t_0\}$. Define a norm-one linear operator $P: A(\Gamma, X) \rightarrow A(\Gamma, X)$ by $Pf = \phi \cdot f$ and denote $Y$ as the restriction of $A(\Gamma, X)$ to $\tilde{\Gamma}$. The operator $P$ is a projection because $Pf(t) = 0$ for all $t \neq t_0$ and $P^2f(t_0) = f(t_0) = Pf(t_0)$. As a matter of fact, the image $P(A(\Gamma, X))$ is isometrically isomorphic to $X$. Indeed, define a linear operator $\Psi: P(A(\Gamma, X)) \rightarrow X$ by $\Psi(Pf) = f(t_0)$. Then $\|\Psi(Pf)\|_X =\|f(t_0)\|_X =  \|Pf\|$. Moreover, we see that for every $x \in X$ there exists $f \in A(\Gamma, X)$ such that $f(t_0) = x$. Hence, $\Psi$ is surjective, which in turn implies that the operator $\Psi$ is an isometric isomorphism on $P(A(\Gamma, X))$.
	
	Now, we claim that the space $A(\Gamma, X)$ is isometrically isomorphic to $X \oplus_{\infty} Y$. For $f \in A(\Gamma, X)$, define a bounded linear operator $\Phi: A(\Gamma, X) \rightarrow X \oplus_{\infty} Y$ by $\Phi f = (Pf, f|_{\tilde{\Gamma}})$. Then we see that
	\[
	\|\Phi f\| = \max\{\|Pf\|, \|f|_{\tilde{\Gamma}}\|\} = \max\left\{\|f(t_0)\|_X, \sup_{t \in\tilde{\Gamma}}\|f(t)\|_X\right\} = \|f\|.
	\]  
	
	Notice that for a given $(f, g) \in X \oplus_{\infty} Y$, there exist $f_1, f_2 \in A(\Gamma, X)$ such that $f = Pf_1$ and $g = f_2|_{\tilde{\Gamma}}$. Let $h = Pf_1 + f_2 - Pf_2 \in A(\Gamma, X)$. From the fact that $h|_{\tilde{\Gamma}} = f_2|_{\tilde{\Gamma}}$, we have $\Phi(h) = (Pf_1, f_2|_{\tilde{\Gamma}})$. Hence, the operator $\Phi$ is also surjective, and so it is an isometric isomorphism between $A(\Gamma, X)$ and $X \oplus_{\infty} Y$. Consequently, in view of Lemma \ref{lem:AKXisom}, the space $A(K, X)$ is isometrically isomorphic to $X \oplus_{\infty} Y$. 
\end{proof}

\subsection{The polynomial Daugavet property in $A(K,X)$}

First we provide a sufficient condition for $A(K, X)$ to have the polynomial Daugavet property. We mention that the proof method is inspired by \cite[Theorem 2.7]{CGKM}.

\begin{Theorem}\label{th:polydaug}
	Let $K$ be a compact Hausdorff space and let $\Gamma$ be the Shilov boundary of the base algebra $A$ of $A(K, X)$. If $\Gamma$ does not have isolated points, then $A(K, X)$ has the polynomial Daugavet property.
\end{Theorem}

\begin{proof}
	 In view of Theorem \ref{th:polydauggen}, it suffices to show that for every $p \in \mathcal{P}(A(K,X))$ with $\|p\| = 1$, every $f_0 \in S_{A(K, X)}$, and every $\epsilon > 0$, there exist $\alpha \in S_{\mathbb{C}}$ and $g \in B_{A(K, X)}$ such that
	\[
	\text{Re}\,\alpha p(g) > 1 - \epsilon\,\,\, \text{and} \,\,\, \|f_0 + \alpha g\| > 2 - \epsilon.
	\]
	
	Let $0 < \epsilon < 1$, $P \in \mathcal{P}(A(K, X))$ with $\|P\| = 1$, and $f_0 \in S_{A(K, X)}$. Choose $h \in S_{A(K, X)}$ and $\alpha \in S_{\mathbb{C}}$ such that $|P(h)| > 1 - \frac{\epsilon}{2}$ and $\text{Re}\,\alpha P(h) > 1 - \frac{\epsilon}{2}$. Also choose $t_0 \in \Gamma_0$ such that $\|f_0(t_0)\|_X > 1 - \frac{\epsilon}{8}$. Let $U = \{t \in K: \|f_0(t) - f_0(t_0)\|_X < \frac{\epsilon}{8}\,\,\, \text{and}\,\,\, \|h(t) - h(t_0)\|_X < \frac{\epsilon}{8}\}$ be a nonempty open subset of $K$. We consider two cases.
	
	Case 1: Suppose that there exists $(t_i)_{i=1}^{\infty} \subset U$ such that $\|\alpha^{-1}f_0(t_i) - h(t_i)\|_X \rightarrow 0$. Then we have
	\begin{eqnarray*}
		\|f_0 + \alpha h\| &\geq& \sup_i \|f_0(t_i) + \alpha h(t_i)\|_X\\
		&\geq& \sup_i(2\|f_0(t_0)\|_X - 2\|f_0(t_0) - f_0(t_i)\|_X - \|f_0(t_i) - \alpha h(t_i)\|_X)\\
		&\geq& 2 - \frac{\epsilon}{4} - \frac{\epsilon}{4} - \frac{\epsilon}{4} > 2 - \epsilon, 
	\end{eqnarray*}
	and we are done with $g = h$.
	
	Case 2: Now suppose that there exists $\eta > 0$ such that $\|\alpha^{-1}f_0(t) - h(t)\|_X > \eta$ for every $t \in U$. Since $\Gamma$ is perfect, we see that the strong boundary point $t_0 \in U$ is not an isolated point. Let $\{U_i\}_{i=1}^{\infty}$ be a collection of pairwise disjoint open subsets of $U$ such that $\cup_{i = 1}^{\infty} U_i \subset U$ and $U _i \cap \Gamma \neq \emptyset$. From the fact that the Choquet boundary $\Gamma_0$ is dense in $\Gamma$, there exist strong boundary points $t_i \in U_i$ for each $i \in \mathbb{N}$. Then by Lemma \ref{lem:urysohnAKX}, there exists $\phi_i \in A$ such that 
	\begin{equation}\label{eq:urysohns}
	\phi_i(t_i) = 1, \,\,\,, \sup_{K\setminus U_i}|\phi_i(t)| < \frac{\epsilon}{2^{i+3}}, \,\,\, \text{and} \,\,\, |\phi_i(t)| + \left(1 - \frac{\epsilon}{2^{i+3}}\right)|1 - \phi_i(t)| \leq 1 \,\,\, \text{for every} \,\,\, t \in K. 
	\end{equation}
	
	Let $h_i = h + \phi_i(\alpha^{-1} f_0(t_i) - h(t_i)) \in A(K, X)$. Then for every $t \in \cup_{i = 1}^{\infty} U_i$, by (\ref{eq:urysohns}), we have 
	\begin{eqnarray*}
		\|h_i(t)\|_X &=& \|h(t) + \phi_i(t)\alpha^{-1}f_0(t_i) - \phi_i(t)h(t_i)\|_X\\
		&\leq& \|h(t) - h(t_i)\|_X + \|h(t_i) - \phi_i(t)h(t_i)\|_X + \|\phi_i(t) \alpha^{-1} f_0(t_i)\|_X\\
		&\leq& \|h(t) - h(t_0)\|_X + \|h(t_0) - h(t_i)\|_X + |1 - \phi_i(t)| + |\phi_i(t)|\\
		&\leq& \frac{\epsilon}{4} + \left(1 - \frac{\epsilon}{2^{i+3}}\right)|1 - \phi(t)| + \frac{\epsilon}{2^{i+3}}|1 - \phi_i(t)| + |\phi_i(t)|\\
		&\leq& \frac{\epsilon}{4} + 1 + \frac{\epsilon}{2^{i+2}} < 1 + \frac{\epsilon}{2}.
	\end{eqnarray*}
	On the other hand, for every $t \in K \setminus \cup_{i=1}^{\infty} U_i$,
	\begin{equation}\label{eq:hit}
	\|h_i(t)\|_X \leq \|h(t)\|_X + |\phi_i(t)| \|\alpha^{-1} f_0(t_i) - h(t_i)\|_X \leq 1 + \frac{\epsilon}{2^{i+3}}\cdot 2 < 1 + \frac{\epsilon}{2}.
	\end{equation}
	Moreover, we see that
	\begin{equation}\label{eq:hi}
	\|h_i\| \geq \|h_i(t_i)\|_X = \|f_0(t_i)\|_X  \geq \|f_0(t_0)\| - \|f_0(t_0) - f_0(t_i)\| \geq 1 - \frac{\epsilon}{2}.
	\end{equation}
	Now, let $g_i = \frac{h_i}{\|h_i\|}$. By (\ref{eq:hi}) we obtain
	\[
	\|h_i - g_i\| = |1 - \|h_i\|| < \frac{\epsilon}{2}.
	\]
	
	For every $(\beta_i) \in \ell_{\infty}$, notice that
	\begin{eqnarray*}
	\sup_n \left\|\sum_{i=1}^n \beta_i \phi_i(\alpha^{-1}f_0(t_i) - h(t_i)) \right\|_X &\leq& \sup_n \sup_{t \in K} \sum_{i=1}^{n}|\beta_i||\phi_i(t)|\|\alpha^{-1}f_0(t_i) - h(t_i)\|_X\\
	&\leq& \sup_n \sup_{t \in K}\sum_{i=1}^n2|\beta_i||\phi_i(t)|\\
	&\leq& 2 \sup_i|\beta_i|\left(1 + \frac{\epsilon}{2^4} + \frac{\epsilon}{2^5}+\cdots\right) = 2 \left(1 + \frac{\epsilon}{8}\right) \sup_i|\beta_i|.
	\end{eqnarray*}
	Hence by \cite[Theorem V.6]{Diest}, the series $\sum_{i=1}^\infty \beta_i \phi_i(\alpha^{-1}f_0(t_i) - h(t_i))$ is weakly unconditionally Cauchy. Since we assumed that $\|\alpha^{-1}f_0(t) - h(t)\|_X > \eta$, there exists a basic subsequence $(\phi_{\sigma(i)}(\alpha^{-1}f_0(t_{\sigma(i)}) - h(t_{\sigma(i)}))$ that is equivalent to the basis $(e_i)$ in $c_0$ by using Bessaga-Pe\l czynski principle \cite[pg. 45]{Diest}. From the fact that a polynomial on a bounded subset of $c_0$ is weakly continuous \cite[Proposition 1.59]{Dn} , we have $\text{Re}\,\alpha P(h_{\sigma(i)})\rightarrow \text{Re}\,\alpha P(h)$ as $i \rightarrow \infty$.
	
	Choose $k \in \mathbb{N}$ such that $\text{Re}\,\alpha P(h_k) > 1 - \frac{\epsilon}{2}$. Then we have
	\[
	\text{Re}\, \alpha P(g_k) = \frac{\text{Re}\, \alpha P(h_k)}{\|h_k\|} > \frac{1 - \epsilon /2}{1 + \epsilon/2} \geq 1 - \epsilon.
	\]
	Therefore, we finally obtain
	\begin{eqnarray*}
		\|f_0 + \alpha g_k\| \geq \|f_0 + \alpha h_k\|- \|g_k - h_k\| &\geq& \|f_0(t_k) + \alpha h_k(t_k)\|_X - \frac{\epsilon}{2}\\
		&=& 2\|f_0(t_k)\|_X - \frac{\epsilon}{2}\\
		&\geq& 2\|f_0(t_0)\|_X - 2\|f_0(t_k) - f_0(t_0)\|_X - \frac{\epsilon}{2}\\
		&\geq& 2\|f_0(t_0)\|_X - \frac{3\epsilon}{4}\\
		&>& 2 - \epsilon,
	\end{eqnarray*}
and we are done with $g = g_k$.
\end{proof}

Let $\mathcal{P}_K(X, X)$ be the set of all compact polynomials from $X$ to itself. For $P \in \mathcal{P}_K(X, X)$, the numerical range $V(P)$ is defined by
\[
V(P) = \{x^*(Px): x^* \in S_{X^*} \,\,\,\text{and} \,\,\, x \in S_X \,\,\,\text{where} \,\,\, x^*(x) = 1\}.
\]  
Now, we recall the polynomial Daugavetian index. 

\begin{Definition}\cite{S}
	For an infinite-dimensional complex Banach space $X$, the polynomial Daugavetian index $\text{Daug}_p\,(X)$ is defined by
	\begin{eqnarray*}
	\text{Daug}_p\,(X) &=& \max\{m \geq 0: \|I + P\| \geq 1 + m\|P\|, \,\,\,\text{for every} \,\,\,P \in \mathcal{P}_K(X, X)\}\\
	&=& \inf\{\omega(P): P \in \mathcal{P}_K(X, X), \|P\| =1\},
	\end{eqnarray*}
	where $\omega(P) = \sup \text{Re}\,V(P)$.
\end{Definition}
\noindent It is well-known that $\text{Daug}_p\,(X) \in [0,1]$ and $\text{Daug}_p(X) \leq \text{Daug}(X)$, where $\text{Daug}(X)$ is the Daugavetian index introduced in \cite{M} considering compact linear operators. A complex Banach space $X$ has the polynomial Daugavet property if and only if  $\text{Daug}_p(X) = 1$. This comes from the fact that a Banach space $X$ satisfies the Daugavet equation for every rank-one polynomials if and only if $X$ satisfies the same equation for every weakly compact polynomials (see Theorem \ref{th:polydauggen}). We recall the following lemma that will be useful later.

\begin{Lemma}\cite[Proposition 2.2, 2.3]{S}\label{lem:familydaug}
Let $\{X_{\lambda}\}_{\lambda \in \Lambda}$ be a family of infinite-dimensional complex Banach spaces and let $Z$ be the $c_0$- or $\ell_{\infty}$-sum of the family. Then
\[
\text{Daug}_p\,(Z) = \inf\{\text{Daug}_p\,(X_\lambda): \lambda \in \Lambda\}.
\]	
\end{Lemma}

If there exists a finite-rank projection on $X$ such that $\|P\| = \|I - P\| = 1$, then $\text{Daug}\,(X) = 0$ \cite{M}. Hence $\text{Daug}_p\,(X) = 0$ in this case. Examples of such spaces are $C(K)$ where $K$ has isolated points and Banach spaces $X$ with an 1-unconditional basis \cite[pp. 635]{M}. 
Similar to the space $C(K)$, we can also construct such a projection for uniform algebras.

	\begin{proposition}\label{prop:bicontractive}
		Let $A$ be a uniform algebra on a compact Hausdorff space $K$ and let $t_0$ be an isolated point of the Shilov boundary $\Gamma$ of $A$. Then, there exists a projection $P:A \rightarrow A$ such that $\|P\| = \|I - P\| = 1$.
	\end{proposition}	

	\begin{proof}
		As shown in the proof of Lemma \ref{lem:akxdecomp} with $X = \mathbb{C}$, the projection $P: A \rightarrow A$ defined by $Pf = \phi \cdot f$ has norm one. For every $f \in S_A$ notice that $\|f - Pf\|_{\infty} = \sup_{t \in \Gamma \setminus \{t_0\}} |f(t)| \leq 1$. 
		
		If we choose a strong boundary point $t_1 \in \Gamma \setminus \{t_0\}$, there exists $g \in A$ such that $\|g\|_{\infty} = g(t_1) = 1$, and so we have $
		\|g - Pg\|_{\infty} = \sup_{t \in \Gamma \setminus \{t_0\}} |g(t)| = 1$.  
		Therefore, $\|P\| = \|I - P\| = 1$.		
	\end{proof}
		
	
		
	
	\begin{Corollary}\label{cor:daugzero}
		Let $K$ be a compact Hausdorff space and let $A$ be a uniform algebra on $K$. If the Shilov boundary of $A$ contains an isolated point, then $\text{Daug}_p(A) = 0$.	
	\end{Corollary}
	
	\begin{proof}
		This is an immediate consequence of Proposition \ref{prop:bicontractive}.
	\end{proof}
    
The following result is inspired by \cite[Proposition 2.4]{S}.
\begin{Theorem}\label{th:polydauind}
	Let $X$ be a complex Banach space and let $K$ be a compact Hausdorff space. Then
	\[
	\text{Daug}_p\,(A(K, X)) = \max\{\text{Daug}_p(A), \text{Daug}_p(X)\}.
	\]
\end{Theorem}

\begin{proof}
	Let $P \in \mathcal{P}_K(A(K, X); A(K, X))$. We first show that 
	\[
	\|I + P\| \geq 1 + \text{Daug}_p(X)\|P\|.
	\]
	For a given $\epsilon > 0$ there exists $f_0\in S_{A(K, X)}$ and $t_0 \in \Gamma_0$ such that $\|P(f_0)(t_0)\|_X \geq \|P\| - \frac{\epsilon}{2}$. Since $P$ is continuous at $f_0$, there exists $\delta > 0$ such that 
	\begin{equation}\label{eq:Pcont}
		\|f_0 - g\| < \delta\,\,\, \text{implies}\,\,\, \|P(f_0) - P(g)\| < \frac{\epsilon}{2}.
	\end{equation}
	
	Now, consider $U = \{t \in K: \|f_0(t) - f_0(t_0)\|_X  <\frac{\delta}{4}\}$. Since the set $U$ is a nonempty open subset of $K$ that contains the strong boundary point $t_0$, by Lemma \ref{lem:urysohnAKX}, there exists $\phi \in A$ such that  
	\begin{equation}\label{eq:urysohn1}
		\phi(t_0) = 1,\,\,\, \sup_{K \setminus U} |\phi(t)| < \frac{\delta}{8},\,\,\, \text{and}\,\,\, |\phi(t)| + \left(1 - \frac{\delta}{8}\right) |1 - \phi(t)| \leq 1 \,\,\, \text{for every} \,\,\, t \in K.
	\end{equation}
	 
	Fix $x_0 \in S_X$ such that $f_0(t_0) = \|f_0(t_0)\|_X \cdot x_0$ and define $\Psi: \mathbb{C} \rightarrow A(K, X)$ by
	\[
	\Psi(z) = \left(1 - \frac{\delta}{8}\right)(1 - \phi)f_0 + \phi \cdot x_0 \cdot z.
	\]
	Then we have
	\begin{eqnarray*}
	\Psi(\|f_0(t_0)\|_X)(t) - f_0(t) &=& \left(1 - \frac{\delta}{8}\right)(1 - \phi(t)) f_0(t) + \phi(t)f_0(t_0) - f_0(t)\\
	&=& \phi(t)(f_0(t_0) - f_0(t)) - \frac{\delta(1- \phi(t))f_0(t)}{8}.
	\end{eqnarray*}
	In view of (\ref{eq:urysohn1}), notice that  
	\[
	\left\|\phi(t)(f_0(t_0) - f_0(t))- \frac{\delta(1- \phi(t))f_0(t)}{8}\right\|_X \leq \|\phi\|_{\infty} \cdot \|(f_0(t_0) - f_0(t))\|_X + \frac{\delta}{4}< \frac{\delta}{2}
	\]
	for every $t \in U$ and that 
	\[
	\left\|\phi(t)(f_0(t_0) - f_0(t))- \frac{\delta(1- \phi(t))f_0(t)}{8}\right\|_X < 2 \cdot \frac{\delta}{4} = \frac{\delta}{2}
	\]
	for every $t \in K \setminus U$. Hence we can see that $\|\Psi(\|f_0(t_0)\|_X) - f_0\| < \delta$, and so
	\[
	\|P(\Psi(\|f_0(t_0)\|_X))(t_0) - P(f_0)(t_0)\|_X \leq \|P(\Psi(\|f_0(t_0)\|_X)) - P(f_0)\| < \frac{\epsilon}{2}
	\] 
	by (\ref{eq:Pcont}). This implies that
	\[
	\|P(\Psi(\|f_0(t_0)\|_X))(t_0)\|_X > \|P(f_0)(t_0)\|_X - \frac{\epsilon}{2}> \|P\| - 2\cdot \frac{\epsilon}{2} = \|P\| - \epsilon.
	\] 
	
	In view of Hahn-Banach theorem, there exists $x_0^* \in S_{X^*}$ such that 
	\[
	x_0^*\left(P(\Psi(\|f_0(t_0)\|_X))(t_0)\right) = \|P(\Psi(\|f_0(t_0)\|_X))(t_0)\|_X  > \|P\| - \epsilon.
	\]
	Notice that the function $f(z) = 	x_0^*\left(P(\Psi(z))(t_0)\right)$ is holomorphic. Hence, by the maximum modulus theorem, there exists $z_0 \in S_{\mathbb{C}}$ such that 
	\[
	\|P(\Psi(z_0))(t_0)\|_X \geq x_0^*\left(P(\Psi(\|f_0(t_0)\|_X))(t_0)\right) > \|P\| - \epsilon.
	\]
	
	Take $x_1 = z_0 x_0 \in S_X$ and let $x_1^* \in S_{X^*}$ such that $x_1^*x_1 = 1$. Define a function $\Phi: X \rightarrow A(K, X)$ by 
	\[
	\Phi(x) = x_1^*x\left(1 - \frac{\delta}{8}\right)(1 - \phi)f_0 + \phi\cdot x.
	\]
	We see that $\|\Phi(x)\| \leq 1$ for every $x \in B_X$ from (\ref{eq:urysohn1}). In particular, $\Phi(x_1) = \Psi(z_0)$. Hence $\|P(\Phi(x_1))(t_0)\|_X > \|P\| - \epsilon$. Consider $Q \in \mathcal{P}_K(X; X)$ defined by $Q(x) = P(\Phi(x))(t_0)$. Notice that 
	\[
	\|Q\| \geq \|Qx_1\|_X = \|(P(\Phi(x_1))(t_0)\|_X > \|P\| - \epsilon.
	\]
	This implies that $\|I + Q\| \geq 1 + \text{Daug}_p(X) \|Q\|  > 1 + \text{Daug}_p(X) (\|P\| - \epsilon)$. Now choose $x_2 \in B_X$ such that $\|x_2 + Qx_2\| > 1 + \text{Daug}_p(X) (\|P\| - \epsilon)$ and let $g = \Phi(x_2)$. Then we obtain
	\begin{eqnarray*}
		\| I + P\| \geq \|g + Pg\| &\geq& \|g(t_0) + P(g)(t_0)\|_X\\
		&\geq& \left\|x_1^*x_2\left(1 - \frac{\delta}{8}\right)(1 - \phi(t_0)f(t_0) + \phi(t_0)x_2 + Q(x_2)\right\|_X\\
		&=& \|x_2 + Q(x_2)\|_X > 1 + \text{Daug}_p(X) (\|P\| - \epsilon).
	\end{eqnarray*}
	As $\epsilon \rightarrow 0$, we have $\|I + P\| \geq 1 + \text{Daug}_p(X)\|P\|$. This consequently shows that $\text{Daug}_p(A(K, X)) \geq \text{Daug}_p(X)$.	
	
	If $\Gamma$ does not have isolated points, then $A(K, X)$ has the polynomial Daugavet property by Theorem \ref{th:polydaug}. This implies that 
	\[
	\text{Daug}_p(A(K, X)) = \text{Daug}_p(A) = 1,
	\]
	and so we have $\text{Daug}_p(A(K, X)) = \max\{\text{Daug}_p(A), \text{Daug}_p(X)\}$.
	
	If $\Gamma$ has isolated points, then $A(K, X) = X \oplus_{\infty} Y$ by Lemma \ref{lem:akxdecomp} and $\text{Daug}(A) = 0$ by Corollary \ref{cor:daugzero}. From Lemma \ref{lem:familydaug}, we see that $\text{Daug}_p (A(K, X)) \leq \text{Daug}_p (X)$. Therefore, we also obtain $\text{Daug}_p (A(K, X)) = \max\{\text{Daug}_p (A), \text{Daug}_p (X)\}$.
\end{proof}

\begin{corollary}\label{cor:polydaugAKX}
	Let $K$ be a compact Hausdorff space. Then the space $A(K, X)$ has the polynomial Daugavet property if and only if either the base algebra $A$ or $X$ has the polynomial Daugavet property.
\end{corollary}

\subsection{Remarks on the property $(D)$ and the convex diametral local diameter two property in $A(K, X)$} 
Since the equivalence between the property ($\mathcal{D}$) and the DLD2P is not clear, it is natural to explore various Banach spaces that potentially distinguish these properties. However, we show that this is not the case for $A(K, X)$. Under the additional assumption that $X$ is uniformly convex, the space $A(K,X)$ has the Daugavet property if and only if the Shilov boundary of the base algebra does not have isolated points \cite[Theorem 5.6]{LT}. Moreover, the Daugavet property of $A(K, X)$ is equivalent to all diametral D2Ps under the same assumption. In fact, carefully inspecting the proof of \cite[Theorem 5.4]{LT}, we see that the rank-one projection constructed in there has norm-one. With the aid of our previous observations, we can see that the DLD2P is also equivalent to the property ($\mathcal{D}$) for $A(K, X)$.


\begin{proposition}\label{prop:dpointvect}\cite[Theorem 5.4]{LT}
	Let $X$ be a uniformly convex Banach space, $K$ be a compact Hausdorff space, $\Gamma$ be the Shilov boundary of the base algebra $A$ of $A(K, X)$, and $f \in S_{A(K, X)}$. Then the following statements are equivalent:
	\begin{enumerate} [\rm(i)]
		\item $f$ is a Daugavet point.
		\item $f$ is a $\Delta$-point.
		\item Every rank-one, norm-one projection $P = \psi \otimes f$, where $\psi \in A(K, X)^*$ with $\psi(f) = 1$, satisfies $\|I - P\| = 2$. 
		\item There is a limit point $t_0$ of $\Gamma$ such that $\|f\| = \|f(t_0)\|_X$. 
	\end{enumerate}
\end{proposition}

\begin{proof}
	(i) $\implies$ (ii) is clear. The implication (ii) $\implies$ (iii) comes from Theorem \ref{prop:deltaequiv}. Indeed, for a Banach space $Y$, a point $f \in S_Y$ is a $\Delta$-point if and only if every rank-one projection of the form $P = \psi\otimes f$, where $\psi \in Y^*$ with $\psi(f) = 1$, satisfies $\|I - P\| \geq 2$. Hence, we immediately have (iii) if this projection $P$ has norm one. (iii) $\implies$ (iv) and (iv) $\implies$ (i) are identical to the proof of (ii) $\implies$ (iii) and (iii) $\implies$ (i) in \cite[Theorem 5.4]{LT}, respectively. 
\end{proof}

\begin{Corollary}\cite[Corollary 5.5]{LT}\label{th:dpoint}
	Let $K$ be a compact Hausdorff space, $\Gamma$ be the Shilov boundary of $A(K)$, and $f \in S_{A(K)}$. Then the following statements are equivalent:
	\begin{enumerate}[\rm(i)]
		\item $f$ is a Daugavet point.
		\item $f$ is a $\Delta$-point.
		\item Every rank-one, norm-one projection $P = \psi \otimes f$, where $\psi \in A(K)^*$ with $\psi(f) = 1$, satisfies $\|I - P\| = 2$.
		\item There is a limit point $t_0$ of $\Gamma$ such that $\|f\|_\infty = |f(t_0)|$.  
	\end{enumerate}			
\end{Corollary}

As a consequence, we obtain the following characterizations for the space $A(K, X)$ and infinite-dimensional uniform algebras. 

\begin{Proposition}\cite[Theorem 5.6]{LT}
	Let $X$ be a uniformly convex Banach space, let $K$ be a compact Hausdorff space, and let $\Gamma$ be the Shilov boundary of the base algebra $A$ of $A(K, X)$. Then the following statements are equivalent:
	\begin{enumerate}[\rm(i)]
		\item $A(K,X)$ has the polynomial Daugavet property.
		\item $A(K,X)$ has the Daugavet property. 
		\item $A(K,X)$ has the DD2P.
		\item $A(K,X)$ has the DLD2P.
		\item $A(K,X)$ has the property ($\mathcal{D}$).
		\item The Shilov boundary $\Gamma$ does not have isolated points.
	\end{enumerate}
\end{Proposition}

\begin{proof}
	(i) $\implies$ (ii) $\implies$ (iii) $\implies$ (iv) $\implies$ (v) is clear from their definitions. Theorem \ref{th:polydaug} shows the implication (vi) $\implies$ (i). Showing (v) $\implies$ (vi) is also identical to the proof of \cite[Theorem 5.6]{LT} with Proposition \ref{prop:dpointvect}. 
\end{proof}

\begin{Corollary}\cite[Corollary 5.7]{LT}\label{lemma:equiDau}
	Let $K$ be a compact Hausdorff space and let $\Gamma$ be the Shilov boundary of a uniform algebra $A(K)$. Then the following are equivalent:
	\begin{enumerate}[\rm(i)]
		\item $A(K)$ has the polynomial Daugavet property.
		\item $A(K)$ has the Daugavet property. 
		\item $A(K)$ has the DD2P.
		\item $A(K)$ has the DLD2P.
		\item $A(K)$ has the property ($\mathcal{D}$).
		\item The Shilov boundary $\Gamma$ does not have isolated points.
	\end{enumerate}
\end{Corollary}

In view of Lemma \ref{lem:urysohnAKX}, we can also show that the sufficient condition for the convex-DLD2P in \cite[Theorem 5.9]{LT} can be described with strong boundary points. 

\begin{Theorem}
	Let $K$ be a compact Hausdorff space, $X$ be a uniformly convex Banach space, and let $\Gamma$ be the Shilov boundary of the base algebra of $A(K, X)$. Denote by $\Gamma'$ the set of limit points of the Shilov boundary. If $\Gamma'\cap \Gamma_0 \neq \emptyset$, then $A(K, X)$ has the convex-DLD2P.
\end{Theorem}

\begin{proof}
	In view of Lemma \ref{lem:AKXisom}, we assume that $K = \Gamma$. Denote the set of all $\Delta$-points of $A(K, X)$ by $\Delta$ and the base algebra of $A(K,X)$ by $A$. We claim that $S_{A(K, X)} \subset \overline{\text{conv}}\Delta$.
	
	Let $f \in S_{A(K, X)}$. Choose a point $t_0 \in \Gamma' \cap \Gamma_0$ and let $\lambda = \frac{1 + \|f(t_0)\|_X}{2}$. For $\epsilon > 0$, let $U$ be an open subset of $K$ such that $\|f(t) - f(t_0)\|_X < \epsilon$. Then by Lemma \ref{lem:urysohnAKX}, there exists $\phi \in A$ such that $\|\phi\|_{\infty} = \phi(t_0) = 1$, $\sup_{t \in K \setminus U}|\phi(t)| < \epsilon$, and
	\[
	|\phi(t)| + (1 - \epsilon)|1 - \phi(t)| \leq 1
	\]
	for every $t \in K$.

Choose a norm-one vector $v_0 \in X$ and let
\[
x_0 = 
\begin{cases} \frac{f(t_0)}{\|f(t_0)\|_X} &\mbox{if } f(t_0) \neq 0 \\
	v_0  &\mbox{if } f(t_0)=0.
\end{cases}
\]
Now, define 
	\begin{align*}
		f_1(t) &= (1 - \epsilon)(1 - \phi(t))f(t) + \phi(t)x_0\\
		f_2(t) &= (1 - \epsilon)(1 - \phi(t))f(t) - \phi(t)x_0, \ \ \ t\in K.
	\end{align*}
Notice that $f_1, f_2 \in A(K, X)$ because $A\otimes X \subset A(K, X)$. Moreover,
\begin{eqnarray*}
\|f_1(t)\|_X &=\left\|(1 - \epsilon)(1 - \phi(t))f(t) + \phi(t) x_0\right\|_X \\&\leq (1 - \epsilon)|1 - \phi(t)| + |\phi(t)| \leq 1,
\end{eqnarray*}
for every $t \in K$. In particular, we have $\|f_1(t_0)\|_X = 1$, and so $ \|f_1(t_0)\|_X = \|f_1\| = 1$. By the same argument, we also have $\|f_2(t_0)\|_X = \|f_2\|= 1$. Thus, $f_1, f_2 \in \Delta$ by Proposition \ref{prop:dpointvect}. Let $g(t) = \lambda f_1(t) + (1 - \lambda)f_2(t)$. We need to consider two cases.

Case 1: Suppose $f(t_0) \neq 0$. Then $g(t)  = (1 - \epsilon)(1 - \phi(t))f(t) + \phi(t)f(t_0)$. We see that 
	\begin{eqnarray*}
		\|g(t) - f(t)\|_X &=& \| (1 - \epsilon)(1 - \phi(t))f(t) + \phi(t)f(t_0) - f(t)\|_X\\
		&=& \| (1 - \epsilon)(1 - \phi(t))f(t) + \phi(t)f(t_0) - (1 -\epsilon)f(t) - \epsilon f(t)\|_X\\
		&=& \|(1-\epsilon)(- \phi(t))f(t) +  (1 - \epsilon)\phi(t)f(t_0) + \epsilon\phi(t)f(t_0) - \epsilon f(t)\|_X\\
		&=&\|(1 - \epsilon)\phi(t)(f(t_0) - f(t)) + \epsilon\phi(t)f(t_0) - \epsilon f(t)\|_X\\
		&\leq& (1 -\epsilon)|\phi(t)|\cdot \|f(t) - f(t_0)\|_X +  \epsilon|\phi(t)|\cdot\|f(t_0)\|_X+ \epsilon \|f(t)\|_X\\
		&\leq & (1 -\epsilon)|\phi(t)|\cdot \|f(t) - f(t_0)\|_X + 2\epsilon.
\end{eqnarray*}
For $t \in U$, we see that $(1 - \epsilon)|\phi(t)| \cdot \|f(t) - f(t_0)\|_X \leq (1 - \epsilon)\epsilon < \epsilon$. On the other hand, for $t \in K \setminus U$, we have $(1 - \epsilon)|\phi(t)| \cdot \|f(t) - f(t_0)\|_X \leq 2(1-\epsilon)\epsilon < 2 \epsilon$. Hence, $\|g - f\| < 4\epsilon$, and so $f \in \overline{conv}\Delta$.

Case 2: Now, suppose $f(t_0) = 0$. Then we have $\|f(t)\|_X < \epsilon$ for every $t \in U$. Moreover, notice that $\lambda = \frac{1}{2}$ and $g(t) = (1 - \epsilon)(1 - \phi(t))f(t)$. This implies that
	\begin{eqnarray*}
		\|g(t) - f(t)\|_X &=& \|(1 - \epsilon)(1 - \phi(t))f(t)  - (1 - \epsilon)f(t) - \epsilon f(t)\|_X\\
		&\leq& (1- \epsilon)|\phi(t)|\cdot \|f(t)\|_X + \epsilon \|f(t)\|_X \leq (1- \epsilon)|\phi(t)|\cdot \|f(t)\|_X + \epsilon.
	\end{eqnarray*}
	Notice that $(1- \epsilon)|\phi(t)|\cdot \|f(t)\|_X \leq (1 - \epsilon)\epsilon < \epsilon$ for every $t \in U$. From the fact that $|\phi(t)| < \epsilon$ for $t \in K \setminus U$, we have $(1- \epsilon)|\phi(t)|\cdot \|f(t)\|_X \leq (1- \epsilon)\epsilon < \epsilon$. This shows that $\|g - f\| < 2 \epsilon$, and so $f \in \overline{\text{conv}}\Delta$.

Since $f \in S_{A(K, X)}$ is arbitrary, we see that $S_X \subset \overline{\text{conv}}\Delta$. Therefore, the space $A(K, X)$ has the convex-DLD2P.
\end{proof}

\begin{corollary}
	Let $K$ be a compact Hausdorff space and $\Gamma'$ be the set of limit points in the Shilov boundary of a uniform algebra. If $\Gamma' \cap \Gamma_0 \neq \emptyset$, then the uniform algebra has the convex-DLD2P.
\end{corollary}

\section{Conclusion}
In this article, we showed that Daugavet property and the diametral local diameter two property on complex Banach spaces can be also studied through rank-one real-linear operators and rank-one real-projections, respectively. Also, strongly locally uniformly alternatively convex or smooth (sluacs) Banach spaces are locally uniformly nonsquare in both real and complex Banach spaces. Hence, sluacs Banach spaces do not contain $\Delta$-points. By using the polynomial Daugavetian index, we also showed that the vector-valued function space $A(K, X)$ has the polynomial Daugavet property if and only if either the underlying base algebra $A$ or the range space $X$ has the polynomial Daugavet property. As a consequence, we obtained that the polynomial Daugavet property, the Daugavet property, the diametral diameter two properties, and the property $(\mathcal{D})$ are equivalent for uniform algebras.

\section{Acknowledgments}
The authors would like to thank the anonymous referees for their valuable comments and suggestions on the article.

\section{Conflict of Interest}
Authors state no conflict of interest.

\section{Funding Information}
The first author was supported by Basic Science Research Program through the National Research Foundation of Korea(NRF) funded by the Ministry of Education, Science and Technology [NRF-2020R1A2C1A01010377]. 	The second author was supported by Basic Science Research Program through the National Research Foundation of Korea(NRF) funded by the Ministry of Education, Science and Technology [NRF-2020R1A2C1A01010377].

\end{document}